\documentclass{article}
\usepackage{epsfig}
\usepackage{amsmath, amssymb, amsthm}
\newcommand{\R}{{I\!\!R}}

\def\R{{\rm I}\! {\rm R}}

\def\ss{sequential splitting}
\def\os{operator-splitting}
\def\X{{\bf X}}

\newtheorem{assumption}{Assumption}[section]
\newtheorem{remark}{Remark}[section]
\newtheorem{theorem}{Theorem}[section]
\newtheorem{corollary}{Corollary}[section]

\begin{document}

\pagestyle{headings}

\title{Multiscale Splitting method for Boltzmann-Poisson Equation:
Application for Dynamic of Electrons}
\author{J\"urgen Geiser
\thanks{University of Greifswald, Institute of Physics, Felix-Hausdorff-Str. 6, D-17489 Greifswald, Germany, E-mail: juergen.geiser@uni-greifswald.de}
 \and Thomas Zacher \thanks{Humboldt Universit\"at zu Berlin, Department of Informatics, Unter den Linden 6, D-10099 Berlin, Germany, E-mail: zacher@informatik.hu-berlin.de} }
\maketitle

\begin{abstract}

In this paper we present a model
based on dynamics of the electrons in the
plasma using a simplified Boltzmann equation
coupled with a Poisson equation.

The motivation arose to simulate active plasma resonance spectroscopy
which is used for plasma diagnostic techniques, see \cite{braith2009}, \cite{lapke2010}
and \cite{oberrath2011}.

We are interested on designing splitting methods to the
model problem.

First we reduce to a simplified transport equation
and start to analyze the abstract Cauchy problem based on semi-groups.

Second we extent to the coupled transport and kinetic model and apply the
splitting ideas.

The results are discussed with first numerical experiments to give 
discuss the numerical methods.

\end{abstract}



\section{Introduction}

We motivate our studying on simulating a active
plasma resonance spectroscopy, which is well established in
plasma diagnostic techniques.

To study the model with simulation models, we concentrate on
an abstract kinetic model, which described the
dynamics of electrons in the plasma by using a 
Boltzmann equation. The Boltzmann equation is coupled
with the electric field and we obtain coupled 
partial differential equations.

Our combined model is done simplified to
apply with functional analytical tools.
We discuss the description of a positive semi-group,
which helps to do the numerical estimations in the
splitting schemes.

Second a numerical method is discussed with respect to
separate differential and integral part of the 
equations.

The numerical approximation is done by applying 
splitting methods of second order.

The paper is outlined as follows.

In section \ref{modell} we present our mathematical model 
and a possible reduced model for the further approximations.

The functional analytical setting with semi-groups
are discussed in section \ref{funct}.
The splitting schemes are presented in in Section \ref{splitt}
and the numerical integration of the integro-part is discussed in 
Section \ref{integro}.

Numerical experiments are done in Section \ref{num}.
In the contents, that are given in Section \ref{concl}, 
we summarize our results.

\section{Mathematical Model}
\label{modell}

In the following a model is presented due to the 
motivation in  \cite{braith2009}, \cite{lapke2010}
and \cite{oberrath2011}.

The models consider a fluid dynamical approach
of the natural ability of plasmas to resonate
in the near of the electron plasma frequency $\omega_{pe}$.

Here we specialize to an abstract kinetic model to
describe the dynamics of the electrons in the plasma, that
allows to do the resonance-analysis.

The Boltzmann equation for the electron particles  are given as
\begin{eqnarray}
\label{neutron}
&& \frac{\partial f(x,v,t)}{\partial t} = - v \cdot \nabla_x f(x,v,t) - \frac{e}{m_e} \nabla_x \phi \cdot \nabla_v f(x,v,t) \nonumber \\
&&  - \sigma(x,v,t)  f(x,v,t) + \int_V \kappa(x,v,v') f(x, v', t) \; dv' , \\
&& f(x,v,0) = f_0(x,v) , 
\end{eqnarray}
and boundary conditions are postulated at the boundaries of $P$ (plasma).

In front of the materials we assume complete reflection of the
electrons due to the sheath $f(v_{||} + v_{\perp})$ with $v_{||}$
is the parallel and $v_{\perp}$ perpendicular to the surface normal
vector. $\phi$ is the electric field.

The Boltzmann's equation has to be coupled with the electric field.
The electrostatic approximation of the field is represented
by the potential that is valid on the complete volume $S$.

We apply the the Poisson's equation:
\begin{eqnarray}
\label{neutron}
&& - \nabla_x \cdot ( \epsilon \nabla \phi ) = \left\{
\begin{array}{c c}
e (n_i - \int \; f \;  dS)  & \mbox{in}  \; P \\
0 & \mbox{in} D 
\end{array} ,
\right.
\end{eqnarray}
the permittivity is equal to $\epsilon_0$ in the plasma $P$
and $\epsilon_0 \epsilon_D$ in the dielectric $D$.
$\phi$ fulfills the boundary conditions $U_n$ at any electrode $E_n$
and $- { \bf n} \cdot \nabla \phi = 0$ at isolator $I$,
whereas ${\bf n}$ is the normal vector of the isolator surface.

On the surface of the dielectric a surface charge $\sigma$ may
accumulate and leads to a transition condition:
\begin{eqnarray}
\label{neutron}
&& \Delta (\epsilon \nabla \phi) = - \sigma.
\end{eqnarray}

\section{Semi-groups for Transport Equations}
\label{funct}

In the following, we derive the exponential growth of the transport semi-groups
that is used in the section of the numerical methods.

We discuss in the following subsections two directions
of the transport regimes:
\begin{itemize}
\item Neutron transport and
\item Electron transport.
\end{itemize}

\subsection{Transport model for the neutrons}

For this model we can assume 
that $f(x, v, t)$ describe the density distribution of particles
at position $x \in S$ with speed $v \in V$ at time $t \in [0, T]$,
see also \cite{EN00} and \cite{rhandi2002}.

The space $S$ is assumed to be a compact and convex subset of $\R^3$
with nonempty interior, and the velocity space $V$ is:

 $V := \{ v \in \R^3: v_{min} \le || ||_{2} \le v_{max} \}$

for $v_{min} > 0 $ and $v_{max} < \infty$.

\begin{assumption}
We have the following assumptions:

\begin{itemize}
\item Particles move according to their speed $v$.
\item Particles are absorbed with function $\sigma$ (e.g. probability function), depending on $x$ and $v$.
\item Particles are scattered to a scattering kernel $\kappa$ depending on position $x$, incoming speed $v'$ and outgoing speed $v$.
\end{itemize}

\end{assumption}

The neutron transport is given as:
\begin{eqnarray}
\label{neutron}
&& \frac{\partial f(x,v,t)}{\partial t} = - v \cdot \nabla f(x,v,t) - \sigma(x,v,t)  f(x,v,t) \nonumber \\
&& + \int_V \kappa(x,v,v') f(x, v', t) \; dv' , \\
&& f(x,v,0) = f_0(x,v) ,
\end{eqnarray}
and boundary conditions are included in the transport operator $A_0$ see in the
following abstract Cauchy problem.

In the following we deal with the abstract Cauchy Problem for the
simplified model.

\subsubsection{Abstract Cauchy problem: Transport model for the neutrons}

We have a Banach space $\X:= L^1(S \times V)$ with Lebesgue measure on $S \times V \subset \R^6$ and define the abstract Cauchy problem as:
\begin{eqnarray}
\label{acp_1}
&& \frac{d u(t)}{dt} = B u(t) \; , \\
&& \frac{d u(t)}{dt} = (A_0 - M_{\sigma} + K_{\kappa}) u(t) \; , \\
&& u(0) = u_0,
\end{eqnarray}
where $u \in \X$.

We have the following operators:

1.) Collision-less transport operator

2.) Absorption operator 

3.) Scattering Operator

An important results for further numerical analysis is the
fact, that the transport semi-group can be estimated by an exponential 
growth, see \cite{EN00}:

\begin{corollary}

We assume that $s(B) > - \infty$ is a dominant eigenvalue and $(S(t))_{t \ge 0}$is irreducible
Then the transport semi-group $(S(t))_{t \ge 0}$ has balanced exponential
growth.
There exists a one-dimensional projection $P$ satisfying $0 < P f$ whenever $0 < f$ such that:
\begin{eqnarray}
\label{acp_2}
&& || \exp(- s(B) t) S(t) - P || \le M \exp(- \epsilon t) ,
\end{eqnarray}
for all $t \ge 0$ and appropriate $M \ge 1$ and $\epsilon > 0$.
\end{corollary}

\subsection{Transport model for the electrons or ions}

For this model we can assume 
that $f(x, v, t)$ describe the density distribution of particles
at position $x \in S$ with speed $v \in V$ at time $t \in [0, T]$,
see also \cite{EN00} and \cite{rhandi2002}.

The space $S$ is assumed to be a compact and convex subset of $\R^3$
with nonempty interior, and the velocity space $V$ is:

 $V := \{ v \in \R^3: v_{min} \le || ||_{2} \le v_{max} \}$

for $v_{min} > 0 $ and $v_{max} < \infty$.

\begin{assumption}
We have the following assumptions:

\begin{itemize}
\item Particles move according to their speed $v$.
\item Particles are absorbed with function $\sigma$ (e.g. probability function), depending on $x$ and $v$.
\item Particles are scattered to a scattering kernel $\kappa$ depending on position $x$, incoming speed $v'$ and outgoing speed $v$.
\item Particles are influenced by the static electric field $\phi$, which can be derived by the kinetic theory.
\end{itemize}

\end{assumption}

The electron transport is given as:
\begin{eqnarray}
\label{neutron}
&& \frac{\partial f(x,v,t)}{\partial t} = - v \cdot \nabla_x f(x,v,t) - \frac{e}{m_e} \nabla_x \phi \cdot \nabla_v f(x,v,t) \nonumber \\
&&  - \sigma(x,v,t)  f(x,v,t) + \int_V \kappa(x,v,v') f(x, v', t) \; dv' , \\
&& f(x,v,0) = f_0(x,v) , 
\end{eqnarray}
and boundary conditions are included in the transport operators. 
$\phi$ is the electric field.

Further we have the Poisson's equation:
\begin{eqnarray}
\label{neutron}
&& - \nabla_x \cdot ( \epsilon \nabla \phi ) = \left\{
\begin{array}{c c}
e (n_i - \int \; f \;  dS)  & \mbox{in}  \; P \\
0 & \mbox{in} D 
\end{array} ,
\right.
\end{eqnarray}
the permittivity is equal to $\epsilon_0$ in the plasma $P$
and $\epsilon_0 \epsilon_D$ in the dielectric $D$.

\begin{eqnarray}
\label{neutron}
&& \frac{\partial f(x,v,t)}{\partial t} = - v \cdot \nabla_x f(x,v,t) - \nabla_x D \cdot \nabla_x f(x,v,t) \nonumber \\
&&  - \sigma(x,v,t) f(x,v,t) + \int_V \kappa(x,v,v') f(x, v', t) \; dv' , \\
&& f(x,v,0) = f_0(x,v) , 
\end{eqnarray}
and boundary conditions are included in the transport operators $A_0$ and $A_1$ see in the
following abstract Cauchy problem.
$D$ is the diffusion parameter that includes the electric 
field.

Next we deal with the abstract Cauchy Problem for the
simplified model.

\subsubsection{Abstract Cauchy problem: Transport model for the neutrons}

We have a Banach space $\X:= L^1(S \times V)$ with Lebesgue measure on $S \times V \subset \R^6$ and define the abstract Cauchy problem as:
\begin{eqnarray}
\label{acp_1}
&& \frac{d u(t)}{dt} = B u(t) \; , \\
&& \frac{d u(t)}{dt} = (A_0 + A_1 - M_{\sigma} + K_{\kappa}) u(t) \; , \\
&& u(0) = u_0,
\end{eqnarray}
where $u \in \X$.

We have the following operators:

1.) Collision-less transport operator

2.) Diffusion operator

3.) Absorption operator 

4.) Scattering Operator

An important results for further numerical analysis is the
fact, that the transport semi-group can be estimated by an exponential 
growth.

\begin{corollary}

We assume that $s(B) > - \infty$ is a dominant eigenvalue and $(S(t))_{t \ge 0}$is irreducible
Then the transport semi-group $(S(t))_{t \ge 0}$ has balanced exponential
growth.
There exists a one-dimensional projection $P$ satisfying $0 < P f$ whenever $0 < f$ such that:
\begin{eqnarray}
\label{acp_2}
&& || \exp(- s(B) t) S(t) - P || \le M \exp(- \epsilon t) ,
\end{eqnarray}
for all $t \ge 0$ and appropriate $M \ge 1$ and $\epsilon > 0$.
\end{corollary}

In the next section we discuss the splitting schemes.

\section{Splitting schemes}
\label{splitt}

The operator-splitting methods are used to solve complex models in
the geophysical and environmental physics, they are developed and
applied in \cite{stra68}, \cite{verwer98} and \cite{zla95}. This
ideas based in this article are solving simpler equations with
respect to receive higher order discretization methods for the
remain equations. For this aim we use the operator-splitting
method and decouple the equation as follows described.

\subsection{Splitting methods of first order for linear equations}

First we describe the simplest \os , which is called {\it \ss} for
the following system of ordinary linear differential equations:
\begin{eqnarray}
\label{gleich_kap31} && \partial_t c(t) = A \; c(t) \; + \; B \;
c(t) \;  ,
\end{eqnarray}
whereby the initial-conditions are $c^n = c(t^n)$. The operators
$A$ and $B$ are spatially  discretized operators, e.g. they
correspond to the discretized in space  convection and diffusion
operators (matrices). Hence, they can be considered as bounded
operators.

The sequential  \os \  method is introduced as a method which
solve the two sub-problems  sequentially, where the different
sub-problems are connected via  the initial conditions. This means
that we replace the original problem (\ref{gleich_kap31}) with the
sub-problems
\begin{eqnarray}
\label{gleich_kap3_3}
&& \frac{\partial c^*(t)}{\partial t} = A c^*(t) \; , \quad \mbox{with} \;
c^*(t^n) = c^n \; , \\
&& \frac{\partial c^{**}(t)}{\partial t} = B c^{**}(t) \; , \quad
\mbox{with} \; c^{**}(t^n) = c^*(t^{n+1}) \; , \nonumber
\end{eqnarray}
whereby the splitting time-step is defined as  $\tau_n = t^{n+1} -
t^n$. The approximated split solution is defined as  $c^{n+1} =
c^{**}(t^{n+1})$.

Clearly, the change of the original problems with the sub-problems
usually results some error, called {\it splitting error}.
Obviously, the splitting  error of the \ss \ method can be derived
as follows (cf. e.g.\cite{gei_2011})
\begin{eqnarray}
\label{gleich_kap3_6} \rho_n & = & \frac{1}{\tau} (\exp(\tau_n
(A+B)) - \exp(\tau_n B) \exp(\tau_n
A) ) \; c(t^n) \nonumber \\
&& \nonumber \\
& = & \frac{1}{2} \tau_n [A, B] \; c(t^n) + O(\tau^2) \; .
\end{eqnarray}
whereby $[A, B] := A B - B A$ is the commutator of $A$ and $B$.
Consequently, the splitting error is  $O(\tau_n)$ when the
operators $A$ and $B$ do not commute, otherwise the method is
exact. Hence, by definition, the \ss \ is called {\it first order
splitting method }.

\subsection{Sequential splitting method  for non-linear problems}

We could use the result for the general formulation
of nonlinear ordinary differential equations:
\begin{eqnarray}
\label{gleich_kap3_7}
&& c'(t) = F_1(t,c(t)) + F_2(t, c(t)) \; ,
\end{eqnarray}
where the initial-conditions are given as $c^n = c(t^n)$.

As before, we can  decouple the above problem into two (typically
simpler) sub-problems, namely
\begin{eqnarray}
\label{gleich_kap3_8}
&& \frac{\partial c^*(t)}{\partial t} = F_1(t,c^*(t)) \; \mbox{with} \; t^n \le t \le
t^{n+1}  \; \mbox{and} \; c^*(t^n) = c^n \; ,  \\
&& \frac{\partial c^{**}(t)}{\partial t} = F_2(t,c^{**}(t)) \; \mbox{with} \; t^n \le t \le
t^{n+1}  \; \mbox{and} \; c^{**}(t^n) = c^*(t^{n+1}) \; ,
\end{eqnarray}
where the initial-values are given as $c^n = c(t^n)$ and the split
approximation on the next time level is defined as  $c^{n+1} =
c^{**}(t^{n+1})$.

For this case the splitting error can be defined by use of the
Jacobians of the non-linear mappings $F_1$ and $F_2$, namely as
\begin{eqnarray}
\label{gleich_kap3_9}
&& \rho_n = \frac{1}{2} \tau [\frac{\partial F_1}{\partial c} F_2,
    \frac{\partial F_2}{\partial c} F_1](t^n, c(t^n)) + {\mathcal O}(\tau_n^2) \; .
\end{eqnarray}
Hence, for the general case the splitting error has of first
order, i.e.  $O(\tau_n)$.

\subsection{Higher order splitting methods for linear operators}

So far we defined the \ss \ which has first order accuracy.
However in the practical computations  in many cases we require
splittings of higher order accuracy.

\subsubsection{Symmetrically weighted sequential splitting.}

In the following we introduce a weighted sequential splitting
method, which is based on two sequential splitting methods with
different ordering of the operators. I.e. we consider again the
Cauchy problem (\ref{gleich_kap31}) and we define the \os \ on the
time interval  $[t^n, t^{n+1}] $ (where $t^{n+1}=t^n+\tau_n$) as
follows

\begin{eqnarray}
\label{gleich_kap3_3}
&& \frac{\partial c^*(t)}{\partial t} = A c^*(t) \; , \quad \mbox{with} \;
c^*(t^n) = c^n \; , \\
&& \frac{\partial c^{**}(t)}{\partial t} = B c^{**}(t) \; , \quad \mbox{with}
\; c^{**}(t^n) = c^*(t^{n+1}) \; . \nonumber
\end{eqnarray}

and
\begin{eqnarray}
\label{gleich_kap3_3} && \frac{\partial v^*(t)}{\partial t} = B
v^*(t) \; , \quad \mbox{with} \;
v^*(t^n) = c^n \; , \\
&& \frac{\partial v^{**}(t)}{\partial t} = A v^{**}(t) \; , \quad
\mbox{with} \; v^{**}(t^n) = v^*(t^{n+1}) \; . \nonumber
\end{eqnarray}
where $c^n$ is known.

Then the approximation at the next time-level $t^{n+1}$ is defined
as
\begin{eqnarray}
\label{gleich_kap3_3}
&& c^{n+1} = \frac{c^{**}(t^{n+1}) +  v^{**}(t^{n+1})}{2}
\end{eqnarray}

The splitting error of this operator splitting method is derived
as follows (cf. \cite{CFH03})
\begin{eqnarray}
\label{gleich_kap3_6} \rho_n & = & \frac{1}{\tau_n} \{\exp(\tau_n
(A+B)) - \frac{1}{2} [\exp(\tau_n B) \exp(\tau_n A)+\exp(\tau_n A)
\exp(\tau_n
B) ]\} \; c(t^n) \nonumber \\
&& \nonumber \\
& = & O(\tau^2) \; .
\end{eqnarray}
An easy computation shows that in general case the splitting error
of this method is  $O(\tau^2)$, i.e. the method is of second order
accurate.  (We note that in the case  of commuting operators $A$
and $B$  the method is exact, i.e. the splitting error vanishes.)

\subsubsection{Strang-Marchuk splitting method.}

One of the most popular and widely used \os s   is the so-called
{\it Strang splitting (or Strang-Marchuk splitting)}, defined as
follows \cite{Mar68,stra68}. The methods reads as follows
\begin{eqnarray}
\label{gleich_kap3_3}
&& \frac{\partial c^*(t)}{\partial t} = A c^*(t) \; , \; \mbox{with} \; t^n \le t \le
t^{n+1/2}  \; \mbox{and} \; c^*(t^n) = c^n \; , \\
&& \frac{\partial c^{**}(t)}{\partial t} = B c^{**}(t) \; ,  \; \mbox{with} \; t^{n} \le t \le
t^{n+1}  \; \mbox{and} \; c^{**}(t^n) = c^*(t^{n+1/2}) \; , \nonumber \\
&& \frac{\partial c^{***}(t)}{\partial t} = A c^{***}(t) \; ,  \; \mbox{with} \;
t^{n+1/2} \le t \le t^{n+1} \; \mbox{and} \; c^{***}(t^{n+1/2}) = c^{**}(t^{n+1}) \; , \nonumber
\end{eqnarray}
where $t^{n+1/2} = t^n +0.5  \tau_n$ and the approximation on the
next time level $t^{n+1}$  is defined as $c^{n+1} =
c^{***}(t^{n+1})$.

The splitting error of the Strang splitting  is
\begin{eqnarray}
\label{gleich_kap3_11} && \rho_n = \frac{1}{24} (\tau_n)^2 ([B,
[B, A]] - 2 [A, [A, B]]) \; c(t^n) + O(\tau_n^4) \; .
\end{eqnarray}
(See, e.g. (\cite{hun96}. ) This means that this \os \ is of
second order, too. (We note that under some special conditions for
the operators $A$ and $B$, the Strang splitting has third order
accuracy and even can be exact \cite{farago04}. )

 In our application the first order splitting for
the convection-reaction- and the diffusion-dispersion-term are
applied, because of the dominance of the space-error. The
time-error for this combination was only a constant in the total
error.

In the next subsection we present the iterative-splitting method.

\subsection{Iterative splitting method}

The following algorithm is based on the iteration with fixed
splitting discretization step-size $\tau$, namely, on the time
interval $[t^n,t^{n+1}]$ we solve the following sub-problems
consecutively  for $i=0,2, \dots 2m$. (Cf. \cite{kan03} and
\cite{glow04}.)

\begin{eqnarray}
 && \frac{\partial c_i(t)}{\partial t} = A
c_i(t) \; + \; B c_{i-1}(t), \;
\mbox{with} \; \; c_i(t^n) = c^{n} \label{gleich_kap33a} \\
&& \mbox{and} \; c_{0}(t^n) = c^n \; , \; c_{-1} = 0.0 , \nonumber
\\\label{gleich_kap33b}
&& \frac{\partial c_{i+1}(t)}{\partial t} = A c_i(t) \; + \; B c_{i+1}(t), \; \\
&& \mbox{with} \; \; c_{i+1}(t^n) = c^{n}\; , \nonumber
\end{eqnarray}
 where $c^n$ is the known split
approximation at the  time level  $t=t^{n}$. The split
approximation at the time-level $t=t^{n+1}$ is defined as
$c^{n+1}=c_{2m+1}(t^{n+1})$. (Clearly, the function $c_{i+1}(t)$
depends on the interval $[t^n,t^{n+1}]$, too, but, for the sake of
simplicity, in our notation we omit the dependence on $n$.)

\smallskip
In the following we will analyze  the convergence and the rate of
the convergence  of the method
(\ref{gleich_kap33a})--(\ref{gleich_kap33b}) for $m$ tends to
infinity for the linear operators $A,B: \!  {\X} \rightarrow {\X}
$ where we assume that these operators and their sum  are
generators of the $C_0$ semi-groups. We emphasize that these
operators aren't necessarily bounded, so, the convergence is
examined in  general Banach space setting.

\smallskip
\begin{theorem}
Let us consider the abstract Cauchy problem in a Banach space \X

\begin{equation}
\begin{array}{c}
{\displaystyle \partial_t c(t) = A c(t) + B c(t), \quad 0 < t
\leq T } \\
\noalign{\vskip 1ex} {\displaystyle c(0)=c_0 }
\end{array} \label{eq:ACP}
\end{equation}

\noindent where  $A,B,A+B: \!  {\X} \rightarrow {\X} $ are given
linear  operators being generators of the $C_0$-semi-group and $c_0
\in \X$ is a given element. Then the iteration process
(\ref{gleich_kap33a})--(\ref{gleich_kap33b}) is convergent  and
the and the rate of the convergence is of second order.
\end{theorem}

\begin{proof}
Let us consider the  iteration
(\ref{gleich_kap33a})--(\ref{gleich_kap33b}) on the  sub-interval
$[t^n,t^{n+1}]$. For the error function $e_i(t)= c(t)-c_i(t)$ we
have the relations

\begin{equation}
\begin{array}{c}
{\displaystyle \partial_t e_i(t) = A e_i(t) + B e_{i-1}(t), \quad
t \in (t^n,t^{n+1}], }\\
\noalign{\vskip 1ex} {\displaystyle e_i(t^n)=0}
\end{array} \label{eq:err1}
\end{equation}
and
\begin{equation}
\begin{array}{c}
{\displaystyle \partial_t e_{i+1}(t) = A e_{i}(t) + B e_{i+1}(t),
\quad
t \in (t^n,t^{n+1}], }\\
\noalign{\vskip 1ex} {\displaystyle e_{i+1}(t^n)=0 }
\end{array} \label{eq:err2}
\end{equation}
for $m=0,2,4, \dots$, with $e_0(0)=0$ and $e_{-1}(t)= c(t)$. In
the following we use the notations $\X^2 $ for the product space
$\X \times \X$ enabled with the norm $\| (u,v) \| = \max \{ \| u
\|, \| v \| \} $ ($u,v \in \X$). The elements  $ {\mathcal E}_i(t)
$, $ {\mathcal F}_i(t) \in \X^2 $ and  the linear operator
${\mathcal A}: \X^2 \rightarrow \X^2 $ are defined as follows
\begin{equation}
\begin{array}{c}
{\displaystyle {\mathcal E}_i(t)= \left[ \begin{array}{cc} e_{i}(t) \\
e_{i+1}(t) \end{array} \right]; \quad {\mathcal F}_i(t)= \left[ \begin{array}{cc} e_{i-1}(t) \\
0 \end{array} \right]; \quad {\mathcal A}= \left[ \begin{array}{cc} A \quad 0 \\
A \quad B \end{array} \right].}
\end{array} \label{eq:defeaf}
\end{equation}
Then, using the notations (\ref{eq:defeaf}), the relations
(\ref{eq:err1})--(\ref{eq:err2}) can be written in the form
\begin{equation}
\begin{array}{c}
{\displaystyle \partial_t {\mathcal E}_i(t) = {\mathcal
A}{\mathcal E}_i(t) + {\mathcal F}_i(t), \quad
t \in (t^n,t^{n+1}], }\\
\noalign{\vskip 1ex} {\displaystyle {\mathcal E}_i(t^n)=0}.
\end{array} \label{eq:erracp}
\end{equation}
Due to our assumptions, ${\mathcal A}$ is a generator of the
one-parameter $C_0$ semi-group $ ({\mathcal A}(t))_{t\ge 0} $,
hence using the variations of constants formula,   the solution of
the abstract Cauchy problem (\ref{eq:erracp}) with homogeneous
initial condition  can be written as
\begin{equation}
\begin{array}{c}
{\displaystyle  {\mathcal E}_i(t) = \int_{t^n}^{t}{\exp ({\mathcal
A}(t-s)) {\mathcal F}_i(s)ds}, \quad t \in [t^n, t^{n+1}].}
\end{array} \label{eq:errsol}
\end{equation}
(See, e.g. \cite{EN00}.)  Hence, using the denotation
\begin{equation}
\begin{array}{c}
{\displaystyle \|{\mathcal E}_i}\|_{\infty} =\sup_{t \in
[t^n,t^{n+1}]}\|{\mathcal E}_i(t)\|\end{array}
\label{eq:norminfty}
\end{equation}
we have
\begin{equation}
\begin{array}{c}
{\displaystyle \| {\mathcal E}_i\|(t) \le  \|{\mathcal
F}_i\|_{\infty} \int_{t^n}^{t}\|{\exp ({\mathcal A}(t-s))\| ds}}=
\\\\ {\displaystyle =
\|{ e}_{i-1}  \|\int_{t^n}^{t}\|{\exp ({\mathcal A}(t-s))\| ds},
\quad t \in [t^n, t^{n+1}].}
\end{array} \label{eq:err1}
\end{equation}
Since $ ({\mathcal A}(t))_{t\ge 0} $ is a semi-group therefore  the
so called {\it growth estimation}

\begin{equation}
\begin{array}{c}
{\displaystyle  \| \exp ({\mathcal A}t)\| \le K \exp (\omega t) };
\quad t\ge 0
\end{array} \label{eq:growth}
\end{equation}
holds with some numbers  $K \ge 0$ and $\omega \in \R$
\cite{EN00}.

\begin{itemize}

\item {Assume that  $ ({\mathcal A}(t))_{t\ge 0} $ is a bounded or
exponentially stable semi-group, i.e. (\ref{eq:growth}) holds with
some $\omega \le 0$. Then obviously the estimate
\begin{equation}
\begin{array}{c}
{\displaystyle  \| \exp ({\mathcal A}t)\| \le K ; \quad t\ge 0 }
\end{array} \label{eq:growth1}
\end{equation}
holds, and, hence on base of (\ref{eq:err1}), we have the relation
\begin{equation}
\begin{array}{c}
{\displaystyle \| {\mathcal E}_i\|(t) \le  K \tau_n
\|{e}_{i-1}\|}, \quad t \in (0,\tau_n).
\end{array} \label{eq:err3}
\end{equation}
}

\item Assume that $ ({\mathcal A}(t))_{t\ge 0} $ has an exponential growth with
 some $\omega > 0$. Using (\ref{eq:err1}) we have
\begin{equation}
\begin{array}{c}
{\displaystyle \int_{t^n}^{t^{n+1}}\|{\exp ({\mathcal A}(t-s))\|
ds} \le K_{\omega}(t), \quad t \in [t^n},{t^{n+1}] },
\end{array} \label{eq:posomega}
\end{equation}
where
\begin{equation}
\begin{array}{c}
 {\displaystyle  K_{\omega}(t) =\frac{K}{\omega}\left(\exp(\omega
 (t-t^n))
 -1 \right), \quad t \in [t^n},{t^{n+1}].}
\end{array} \label{eq:komega}
\end{equation}
Hence
\begin{equation}
\begin{array}{c}
 {\displaystyle  K_{\omega}(t) \le \frac{K}{\omega}\left(\exp(\omega
 \tau_n)
 -1 \right)= K \tau_n + {\mathcal O}(\tau_n^2)}
\end{array} \label{eq:komegaest}
\end{equation}

\end{itemize}

The estimations (\ref{eq:err3}) and (\ref{eq:komegaest}) result in
that
\begin{equation}
\begin{array}{c}
{\displaystyle \| {\mathcal E}_i\|_\infty = K \tau_n
\|{e}_{i-1}\|}+ {\mathcal O }(\tau_n^2).
\end{array} \label{eq:basic1}
\end{equation}
Taking into the account the definition of ${\mathcal E}_i$ and the
norm $\| \cdot \|_\infty$, we obtain \begin{equation}
\begin{array}{c}
{\displaystyle \| e_i\| = K \tau_n \|{e}_{i-1}\|}+ {\mathcal O
}(\tau_n^2),
\end{array} \label{eq:basic2}
\end{equation}
and hence \begin{equation}
\begin{array}{c}
{\displaystyle \| e_{i+1}\| = K_1 \tau_n^2 \|{e}_{i-1}\|}+
{\mathcal O }(\tau_n^3),
\end{array} \label{eq:basic3}
\end{equation}
which proves our statement.

\end{proof}

\begin{remark}{\rm
When $A$ and $B$ are matrices (i.e.
(\ref{gleich_kap33a})--(\ref{gleich_kap33b}) is a system of the
ordinary differential equations), for the growth estimation
(\ref{eq:growth}) we can use the concept of the logarithmic norm.
(See e.g.\cite{HunVer03}.) Hence, for many important class of
matrices we can prove the validity of (\ref{eq:growth}) with
$\omega \le 0.$}
\end{remark}

\begin{remark}{\rm  We note that a huge class of important differential
operators generate contractive semi-group. This means that for such
problems -assuming the exact solvability of the split
sub-problems- the iterative splitting method is convergent in
second order to the exact solution. }
\end{remark}

Modify to :

\begin{eqnarray}
\label{gleich_kap3_3a} && \frac{\partial c^i(t)}{\partial t} = A
c^i(t) \; + \; B c^{i-1}(t), \;
\mbox{with} \; \; c^i(t^n) = c^{i-1}({t^{n+1}}) \\
&& \mbox{and the starting values} \; c^{0}(t^n) = c(t^n) \; \mbox{results of
last iteration} \; , \;
c^{-1}(t^n) = 0.0 , \nonumber \\
&& \frac{\partial c^{i+1}(t)}{\partial t} = A c^i(t) \; + \; B c^{i+1}(t), \; \\
&& \mbox{with} \; \; c^{i+1}(t^n) = c^{i}({t^{n+1}})\; , \nonumber \\
&& \epsilon > | c^{i+1} (t^{n+1}) - c^{i-1}(t^{n+1})| \mbox{Stop criterion} \\
&& \mbox{result for the next time-step} \\
&&  c(t^{n+1}) = c^{m}(t^{n+1}) , \; \mbox{for} \;  m \; \mbox{fulfill the stop-criterion}
\end{eqnarray}
for each $i=0,2, \dots$, where $c^n$ is the known split
approximation at the previous time level.

\section{Numerical Integration of the Integro-Part}
\label{integro}

We deal with the following integro-differential equation:

\begin{eqnarray}
\label{neutron}
&& \frac{\partial u}{\partial t} =  \int_0^t u(s) \;  ds , \\
&& u(0) = u_0 , 
\end{eqnarray}

The integration part is done numerically with:

Trapezoidal rule:

\begin{eqnarray}
 \int_a^b f(x)\,dx \approx \frac{b-a}{n} \left( {f(a) + f(b) \over 2} + \sum_{k=1}^{n-1} f \left( a+k \frac{b-a}{n} \right) \right)
\end{eqnarray}
where the sub-intervals have the form $[k h, (k+1) h]$, with $h = (b−a)/n$ and $k = 0, 1, 2, ..., n−1$.

The higher order formulas are given as closed Newton–Cotes formulas
are given as

\begin{table}[h]
\centering
\begin{tabular}{| c | c | c | c |}
	\hline
  Degree  &	Common name & 	Formula  &	Error term \\
 1  &	Trapezoid rule & 	$\frac{b-a}{2} (f_0 + f_1)$ & 	$-\frac{(b-a)^3}{12}\,f^{(2)}(\xi) $ \\
2 &	Simpson's rule &   	$\frac{b-a}{6} (f_0 + 4 f_1 + f_2)$ & 	$-\frac{(b-a)^5}{2880}\,f^{(4)}(\xi) $ \\
3 & 	Simpson's $3/8$ rule & 	$\frac{b-a}{8} (f_0 + 3 f_1 + 3 f_2 + f_3) $ &	$-\frac{(b-a)^5}{6480}\,f^{(4)}(\xi)$ \\
4 & 	Boole's rule &	$\frac{b-a}{90} (7 f_0 + 32 f_1 + 12 f_2$ &	$-\frac{(b-a)^7}{1935360}\,f^{(6)}(\xi) $  \\
   &  &	$  + 32 f_3 + 7 f_4)$ &	  \\
	\hline
\end{tabular}
\caption{\label{table_1} Numerical Integration formulas (Closed Newton–Cotes Formulas).}
\end{table}
where  $f_i$ is a shorthand for $f(x_i)$, with $x_i = a + i (b - a) / n$, and $n$ the degree.

We obtain the following formulas for the Trapezoidal-rule:
\begin{eqnarray}
\label{neutron}
&& \frac{\partial u}{\partial t} =  t/2 (u(0) + u(t)) \;  ds , \\
&& u(0) = u_0 , 
\end{eqnarray}
and obtain the analytical result:
\begin{eqnarray}
\label{neutron}
&& u(t) =  \frac{2}{2} \exp(\frac{t^2}{4}) u(0) - \frac{1}{2} u(0) ,
\end{eqnarray}

For the higher order formula like Simpsons-rule, we have the following results:
\begin{eqnarray}
\label{neutron}
&& \frac{\partial u}{\partial t} =  t/6 (u(0) + 4 u(t/2) + u(t)) \;  ds , \\
&& u(0) = u_0 , 
\end{eqnarray}

We apply the idea of the polynomial solution:

$u(t) = a_0 + a_1 t + a_2 t^2 + a_3 t^3 + \ldots $

and we obtain the results with deriving the coefficients :
\begin{eqnarray}
\label{neutron}
&& a_1  + 2 a_2 t + 3 a_3 t^2 + \ldots  \\
&& =  t/6 \big( a_0 + 4 ( a_0 + a_1 t/2 + a_2 t^2/4 + a_3 t^3/8 + \ldots) \nonumber \\
&&  +  a_0 + a_1 t + a_2 t^2 + a_3 t^3 + \ldots \big) \;  , \nonumber \\
&& a_0 = u_0 , 
\end{eqnarray}
and we obtain via coefficient comparison:
\begin{eqnarray}
\label{neutron}
&& a_0 = u_0 \\
&& a_1 = a_3 = a_5 = \ldots = 0 \\
&& a_2 = 3 a_0 \\
&& a_4 = \frac{1}{12} a_2 , \ldots , 
\end{eqnarray}

\begin{remark}
Such fast algorithms of generalized Taylor series about a function (here
we apply numerical integration formulas) are computed very efficient,
see also  the decomposition ideas of \cite{adom94}.
\end{remark}

\section{Experiments for the Plasma resonance spectroscopy}
\label{num}

\subsection{First Example: Matrix problem with integral term}

We deal with a simpler integro-differential equations:

\begin{eqnarray}
\label{equation1}
  c'(t) & =& c + \int_0^t c(t) dt, \; t \in [0,1],
\end{eqnarray}
where we assume $ \int_0^t c(t) dt = t c(t)$ as a first
order approximation of the integral and deal
with:
\begin{eqnarray}
\label{equation1}
  c'(t) & =& c + t c(t), \; t \in [0,T], \\
c(0) = c(t_0) = 1 ,
\end{eqnarray}
where $T = 10.0$ and we have the analytical solution 
for the approximation which is given as:
\begin{eqnarray}
\label{equation1}
  c(t) & =& \exp(t + \frac{t^2}{2}) u(0) ,
\end{eqnarray}

We split into:
\begin{eqnarray}
\label{equation1}
  A & = & 1 , \\
  B(t) & = & t ,
\end{eqnarray}

We have the following solutions for the iterative scheme:
\begin{equation}
\begin{array}{c}
 c_1(t) = \exp(A (t^{n+1} - t )) c(t^n) , \quad t \in (t^n,t^{n+1}], \\
\end{array} \label{eq:err1}
\end{equation}
\begin{equation}
\begin{array}{l l}
 c_2(t) & = \exp(\int_{t_n}^{t^{n+1}}B(t) dt ) c(t^n) \\
& + \int_{t^n}^{t^{n+1}} \exp(\int_{s}^{t^{n+1}} B (t^{n+1} - t)) dt A c_1(s) ds , \quad t \in (t^n,t^{n+1}] .
\end{array} \label{eq:err1}
\end{equation}
where $n = 0, 1, \ldots, N$ and $t^N=T$ while the time-steps
are given as $\Delta t = t^{n+1} - t^n$.

We deal with the following recurrence relations with even and odd iterations:

for the odd iterations: $i = 2m+1$,

for $m= 0, 1, 2, \ldots$
\begin{equation}
\label{odd_1}
\begin{array}{l l}
 c_{i}(t) & = \exp(A (t - t^n )) c(t^n) \\
 & + \int_{t^n}^{t} \exp(s A) B(s) c_{i-1}(t^{n+1} - s) \; ds , \quad t \in (t^n,t^{n+1}] .
\end{array}
\end{equation}

For the even iterations: $i = 2m$,

for $m= 1, 2, \ldots$
\begin{equation}
\label{even_1}
\begin{array}{l l}
 c_{i}(t) & = \exp(\int_{t_n}^{t^{n+1}} B (s) ds) c(t^n) \\
 & + \int_{t^n}^{t} \exp(\int_0^s B(t) dt) A c_{i-1}(t^{n+1} - s) \; ds , \quad t \in (t^n,t^{n+1}] .
\end{array}
\end{equation}

In the table \ref{table_1} we obtain the numerical results of the iterative splitting scheme.
\begin{table}[h]
\centering
\begin{tabular}{| c | c | c | c | c | c |}
	\hline
	 & $\Delta$t=1 & $\Delta$t=0.5 & $\Delta$t=0.25 & $\Delta$t=$2^{-3}$ & $\Delta$t=$2^{-4}$ \\
	\hline
	$c_1$ & $1.7634$ & $0.4958$ & $0.1793$ & $0.0753$ & $0.0343$ \\
	\hline
	$c_2$ & $0.8628$ & $0.1444$ & $0.0282$ & $0.0061$ & $0.0014$ \\
	\hline
	$c_3$ & $0.2220$ & $0.0104$ & $5.2127$e-04$$ & $2.8455$e-05$$ & $1.6511$e-06$$ \\
	\hline
	$c_4$ & $0.1116$ & $0.0041$ & $1.8660$e-04$$ & $9.7846$e-06$$ & $5.5769$e-07$$ \\
	\hline
	$c_5$ & $0.0971$ & $0.0039$ & $1.8367$e-04$$ & $9.7418$e-06$$ & $5.5644$e-07$$ \\
	\hline
	$c_6$ & $0.0956$ & $0.0039$ & $1.8365$e-04$$ & $9.7418$e-06$$ & $5.5644$e-07$$ \\
	\hline
	$c_7$ & $0.0954$ & $0.0039$ & $1.8365$e-04$$ & $9.7416$e-06$$ & $5.5644$e-07$$ \\
	\hline
	$c_8$ & $0.0954$ & $0.0039$ & $1.8366$e-04$$ & $9.7343$e-06$$ & $5.5556$e-07$$ \\
	\hline
\end{tabular}
\caption{\label{table_1} Numerical experiment with 10 iterative steps for the first example.}
\end{table}

In the Figure \ref{example_1}, we present the one-side and two-side
iterative results.
\begin{figure}[ht]
\begin{center}  
\includegraphics[width=9.0cm,angle=-0]{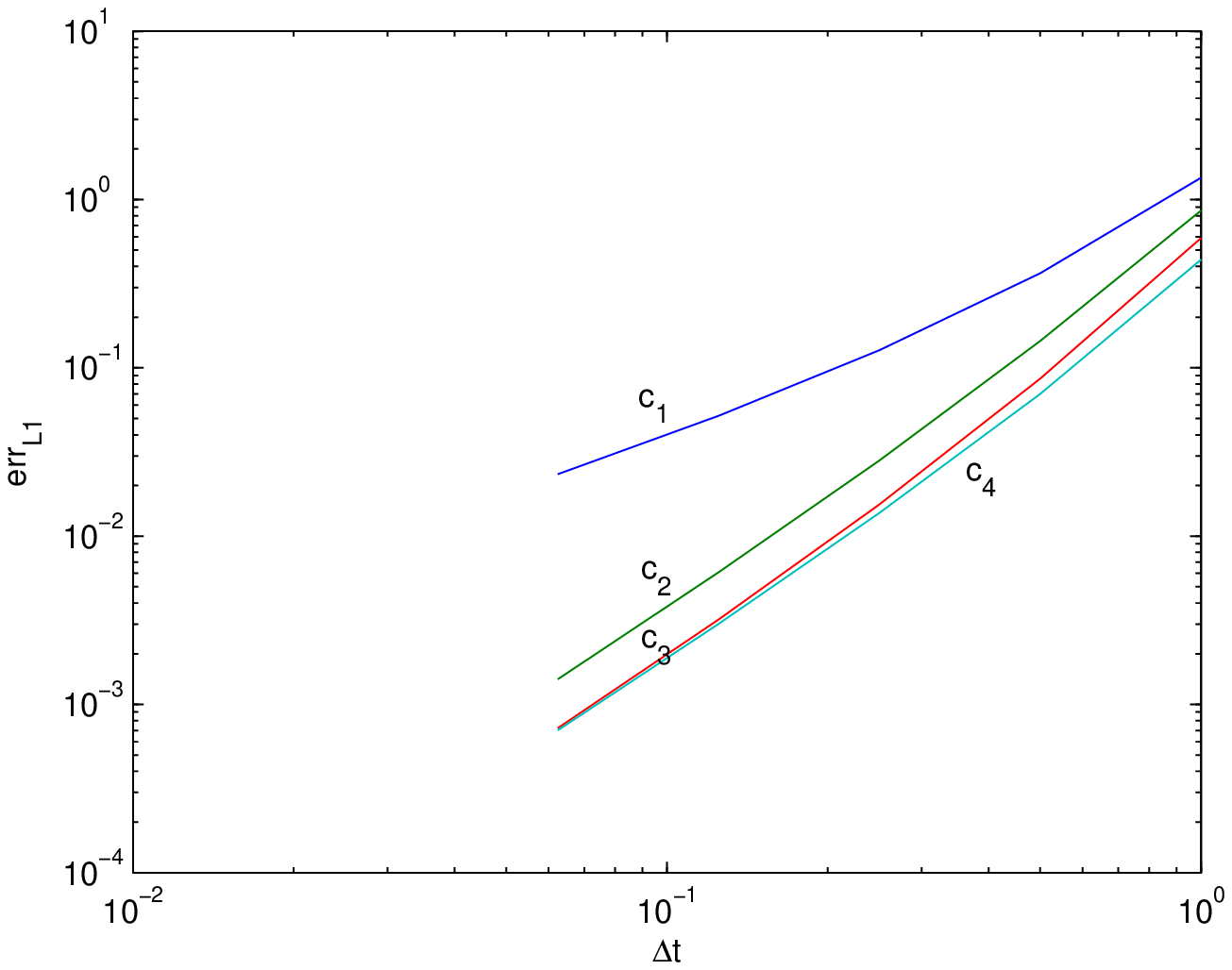} 
\includegraphics[width=9.0cm,angle=-0]{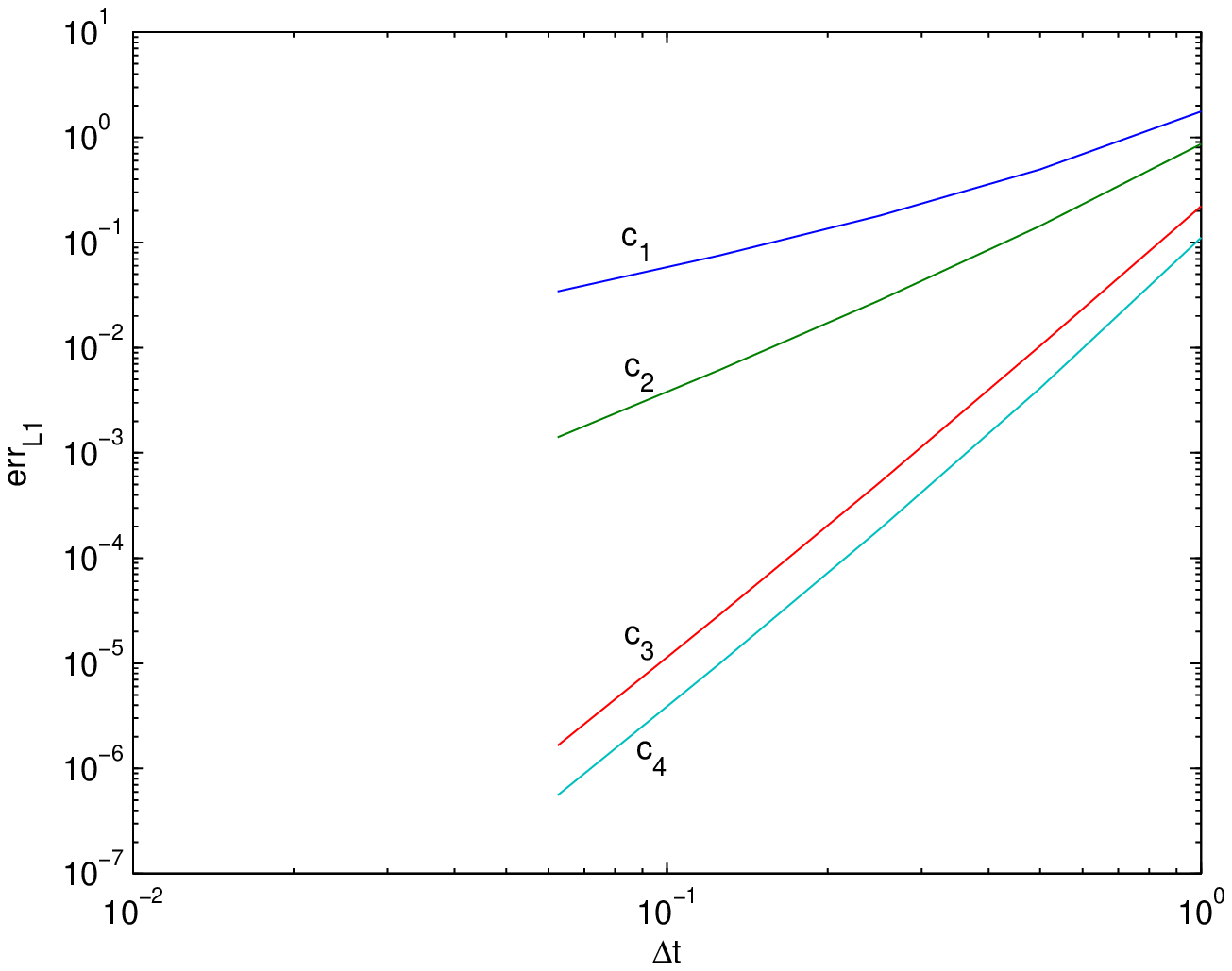} 
\includegraphics[width=9.0cm,angle=-0]{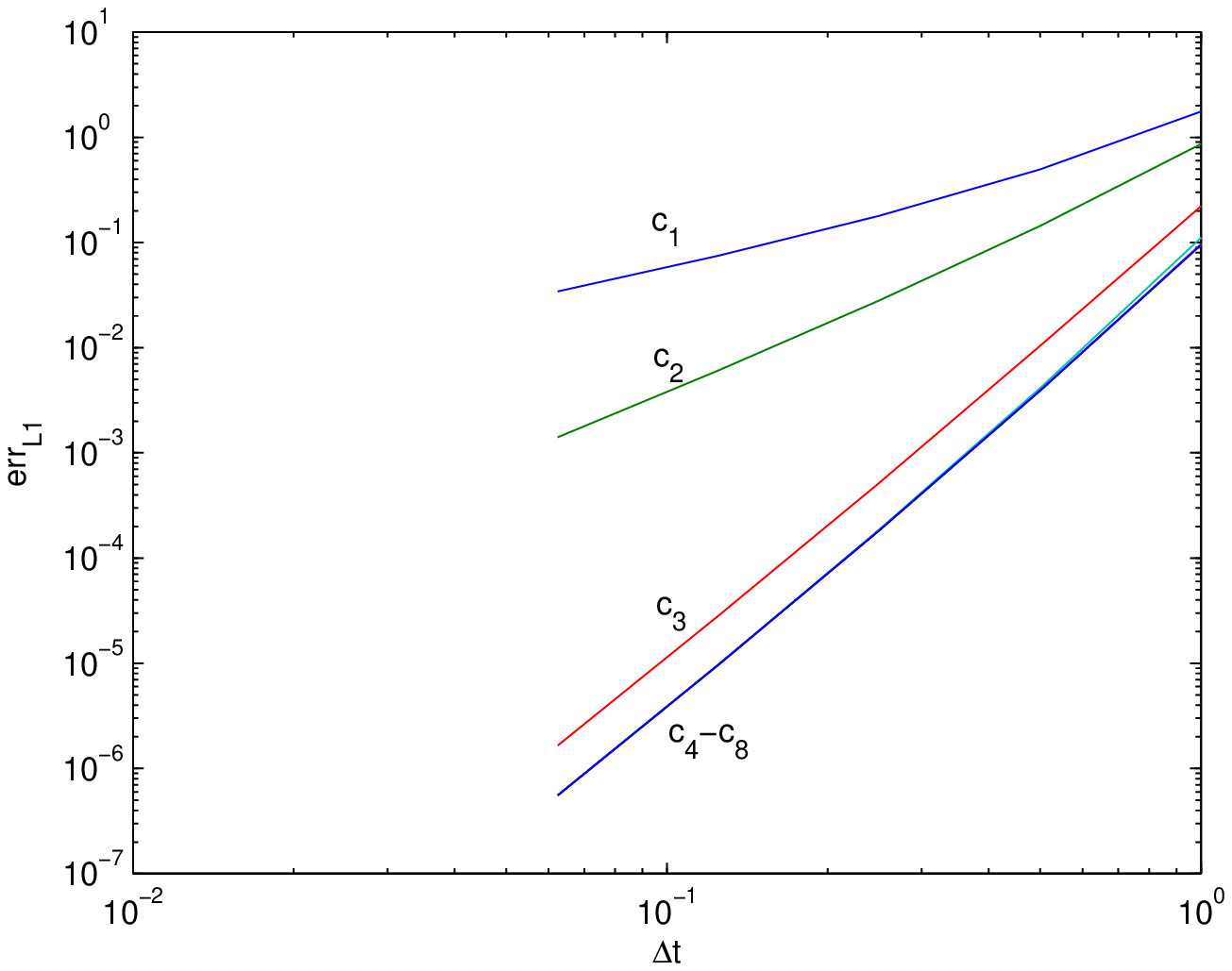} 
\end{center}
\caption{\label{example_1} Numerical errors of the one-side and two-side Splitting scheme:  one-side splitting over $A$ (upper figure), one-side splitting over $B$ (middle figure) and  two-side splitting scheme alternating between $A$ and $B$ (lower figure) with $1, \ldots, 8$ iterative steps.}
\end{figure}

\begin{remark}
In the experiments, we obtain improved results with each additional step.
By the way the solution blows up and we have to use also very fine time-steps to control the errors. 
Optimal results are obtain by using the integral part (stiff part) as the
implicit part in the iteration (one-side over $B$).

\end{remark}

\subsection{Real-life problem}

In the following subsections, we present our experiments
based on the neutron transport.

A simplified  one-dimensional model is given as:
\begin{eqnarray*}
&& \partial_t \; c + v \partial_x c - D \partial_{xx} c + \sigma \; c =  \int_{\Omega} \kappa(x,v,v') c(x,v',t) \; dv' , \\
\end{eqnarray*}
The velocity $v$ and the diffusion $D$ is given by the 
plasma model.
The initial conditions are given by $c(x, 0) = c_0(x)$ 
and the boundary conditions are trivial $\partial_n c(x,t) = 0$.

A first integral operator is given as:
\begin{eqnarray*}
  \int_{\Omega} \kappa(x,v,v') c(x,v',t) \; dv' = \int_0^T c(x,t) dt, \\
\end{eqnarray*}

A second integral operator is given as:

We assume a simple collision operator: $\kappa(x,v,v') = q(v') (1 + v'^2) $

where $q(v')$ is the potential, $e.g. v'^2$. \\

We deal with the first integral operator and define the
following operators:

$A =  v \frac{1}{\Delta x} [1 -1] I - D  \frac{1}{\Delta x^2} [1 -2 1] I $

$ B = (- \sigma + t) I $

while 

$ \exp( B t)  = \exp( (- \sigma t + t^2/2) I) $

where $I$ is the identity matrix of rank.

\subsubsection{One phase example}

The next example is a simplified real-life problem 
for a neutron transport equation,
which includes the gain and loss of a neutron

We concentrate on the computational benefits of a fast
computation of the iterative scheme, given with matrix exponentials. \\

The equation is given as:
\begin{eqnarray}
\label{mobile1_1_2}
&& \partial_t c  + \nabla \cdot {\bf F} c  = - \lambda_1 c + \int_0^t \lambda_2 c(x,t) dt ,  \; \mbox{in} \; \Omega \times [0, t] , \\
&& {\bf F}  = {\bf v}  - D \nabla , \\
&& c({\bf x}, t) = c_{0}({\bf x}),  \; \mbox{on} \; \Omega , \\
&& c({\bf x}, t) = c_{1}({\bf x},t) , \;  \mbox{on} \; \partial \Omega  \times [0, t] ,
\end{eqnarray}

In the following we deal with the semi-discretized equation given with the 
matrices:
\begin{eqnarray}
\label{eq20}
&& \partial_t {\bf C} = 
\left( A - \Lambda_1 + \Lambda_2 \right)  {\bf C} , 
\end{eqnarray}
where $ {\bf C} = (c_{1}, \ldots, c_{I})^T$ is the solution of
the species in the mobile phase in each spatial discretization point (i = 1, \ldots, I).

We have the following two operators for the splitting method:
\begin{eqnarray}
A & = &  \frac{D}{\Delta x^2}\cdot  \left(\begin{array}{rrrrr}
 -2 & 1 & ~ & ~ & ~ \\
 1 & -2 & 1 & ~ & ~ \\
 ~ & \ddots & \ddots & \ddots & ~ \\
 ~ & ~ & 1 & -2 & 1 \\
 ~ & ~ & ~ & 1 & -2
\end{array}\right) \\[6pt]
 & + & \frac{v}{\Delta x}\cdot \left(\begin{array}{rrrrr}
 1 & ~ & ~ & ~ & ~ \\
 -1 & 1 & ~ & ~ & ~ \\
 ~ & \ddots & \ddots & ~ & ~ \\
 ~ & ~ & -1 & 1 & ~ \\
 ~ & ~ & ~ & -1 & 1
\end{array}\right)~\in~\R^{I \times I}
\end{eqnarray}
where $I$ is the number of spatial points.
\begin{eqnarray}
\Lambda_1 & = &    \left(\begin{array}{rrrrr}
 \lambda_1 & 0 & ~ & ~ & ~ \\
        0 &  \lambda_1 & 0 & ~ & ~ \\
 ~ & \ddots & \ddots & \ddots & ~ \\
 ~ & ~ & 0 & \lambda_1 & 0 \\
 ~ & ~ & ~ & 0 & \lambda_1
\end{array}\right) ~\in~\R^{I \times I}
\end{eqnarray}
For the integral term we have the following ideas:

{\bf Case 1:}

$ \int_0^t \lambda_2 c(x,t) dt \approx \lambda_2 t c(x,t)$

and we obtain the Matrix:
\begin{eqnarray}
\Lambda_2 & = &    \left(\begin{array}{rrrrr}
 \lambda_2 t^2/2& 0 & ~ & ~ & ~ \\
        0 &   \lambda_2 t^2/2 & 0 & ~ & ~ \\
 ~ & \ddots & \ddots & \ddots & ~ \\
 ~ & ~ & 0 &  \lambda_2 t^2/2 & 0 \\
 ~ & ~ & ~ & 0 &  \lambda_2 t^2/2
\end{array}\right) ~\in~\R^{I \times I}
\end{eqnarray}
For the operator splitting scheme we apply $A $ and $B = - \Lambda_1 + \Lambda_2$ and we apply the
iterative splitting method, given in equations (\ref{odd_1})- (\ref{even_1}).

{\bf Case 2:}

We integrate the operator B with respect to the previous solutions ${\bf C}_{i-1}$ and we obtain the Matrix:
\begin{eqnarray}
&& \Lambda_2({\bf C}_{i-1}) \nonumber \\
&& =    \left(\begin{array}{c c c c}
 \int_0^t \lambda_2  \; c_{1, i-1}(x,s) \; ds & 0 & \hdots &  0 \\
        0 &   \int_0^t \lambda_2  \; c_{2, i-1}(x,s) \; ds & 0 &   \\
 \vdots  &  \ddots & \ddots & \vdots \\
%  &  & 0 &   \int_0^t \lambda_2  \; c_{I-1, i-1}(x,s) \; ds & 0 \\
0  & \hdots & 0 &  \int_0^t \lambda_2  \; c_{I, i-1}(x,s) \; ds
\end{array}\right) \nonumber \\
&& \in \R^{I \times I}
\end{eqnarray}

We obtain $B({\bf C}) = \Lambda_2({\bf C}_{i-1}) + \Lambda_1 {\bf C} $

The iterative scheme is given as:

For $i=  1, 2, \ldots$
\begin{equation}
\label{odd_2}
\begin{array}{l l}
 {\bf C}_{i}(t) & = \exp(A (t - t^n )) {\bf C}(t^n) \\
 & + \int_{t^n}^{t} \exp((t-s) A) B({\bf C}_{i-1}(s)) \; ds , \quad t \in (t^n,t^{n+1}] .
\end{array}
\end{equation}

For the reference solution, we apply a fine time- and spatial scale without
decoupling the equations.

The Figure \ref{two_phase} present the numerical errors between the exact and the
numerical solution. Here we obtain optimal results for one-side iterative schemes on operator $B$, means we iterate with respect to $B$ and use $A$ as right hand side. 
\begin{figure}[ht]
\begin{center}  
\includegraphics[width=9.0cm,angle=-0]{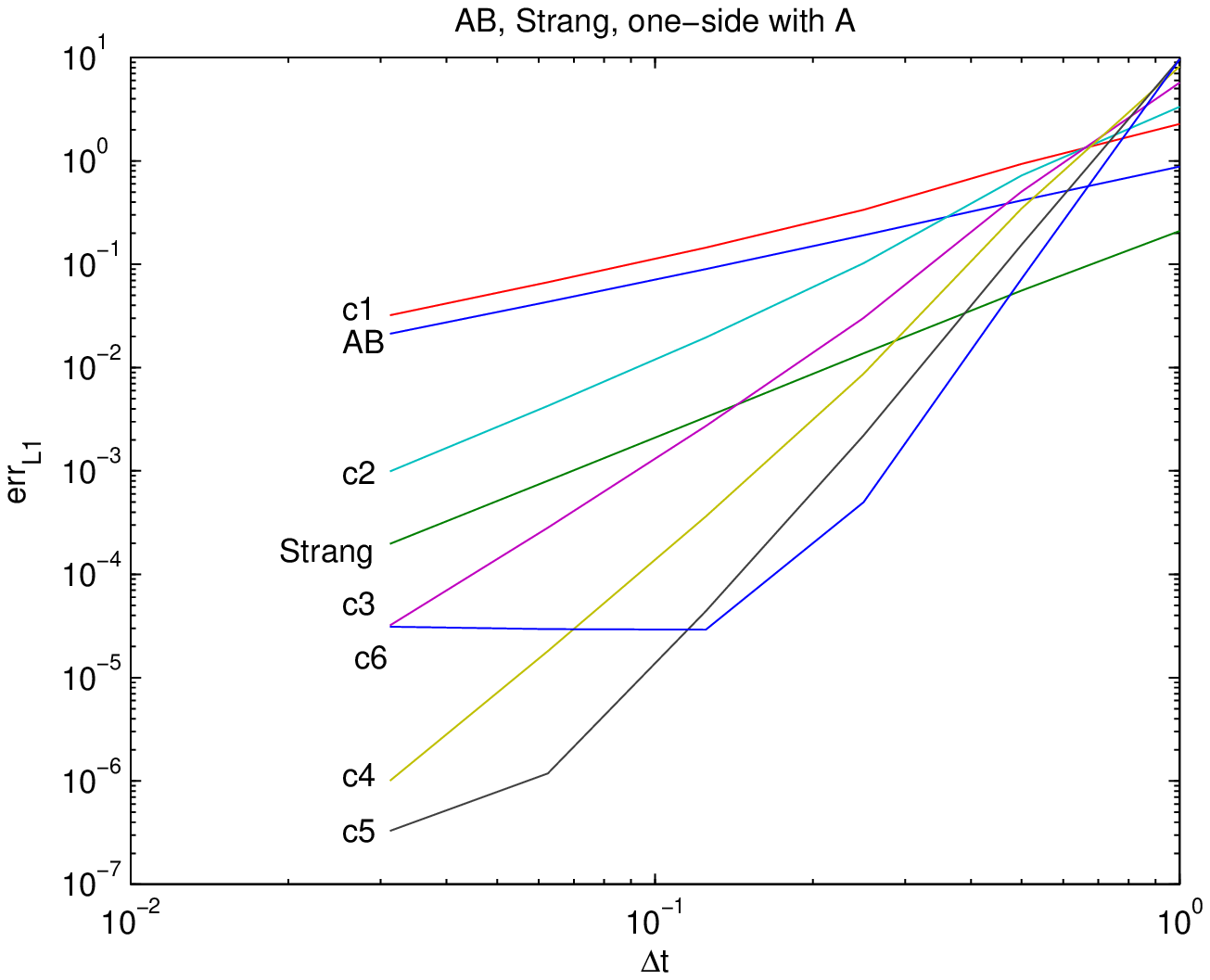} 
\includegraphics[width=9.0cm,angle=-0]{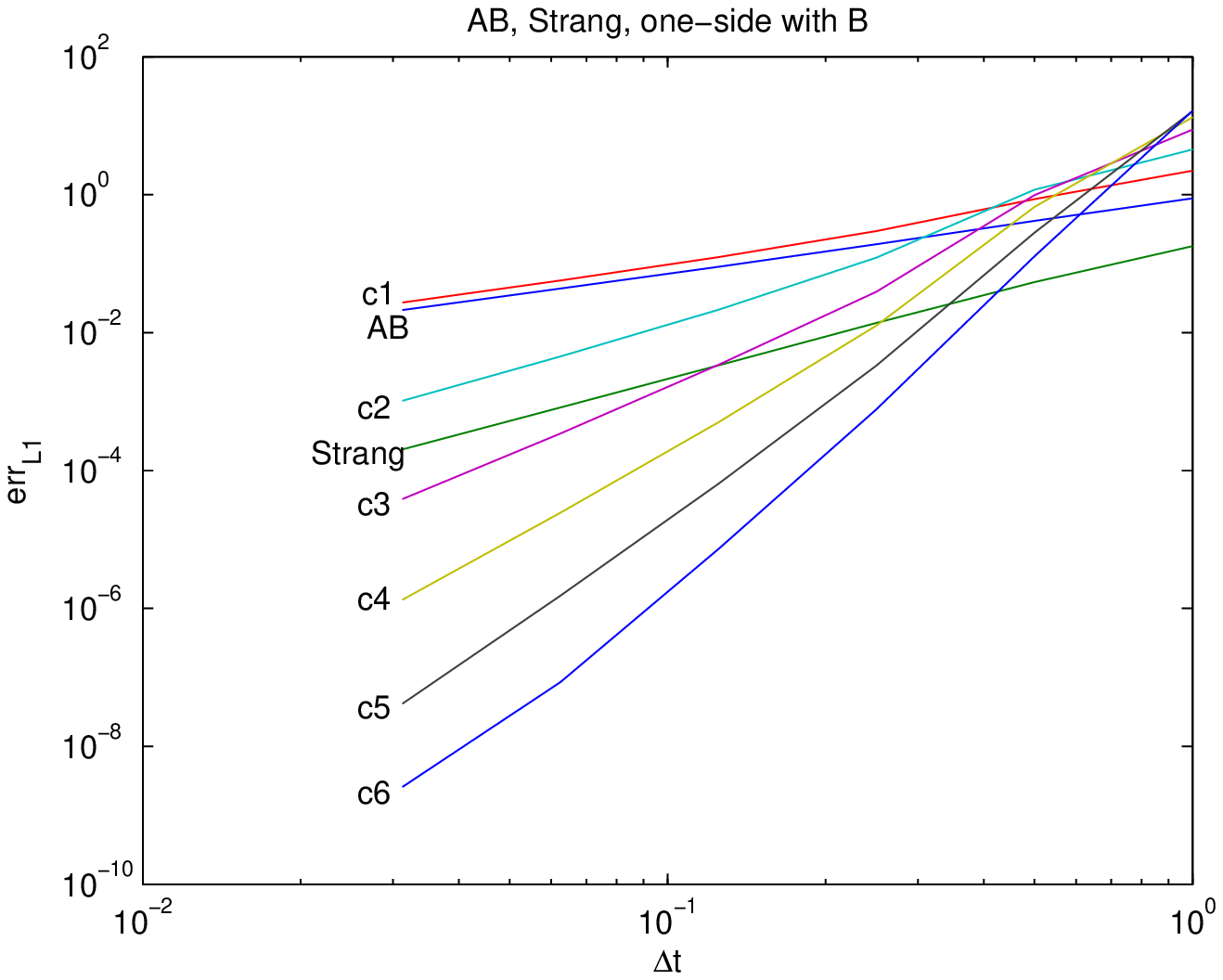} 
\includegraphics[width=9.0cm,angle=-0]{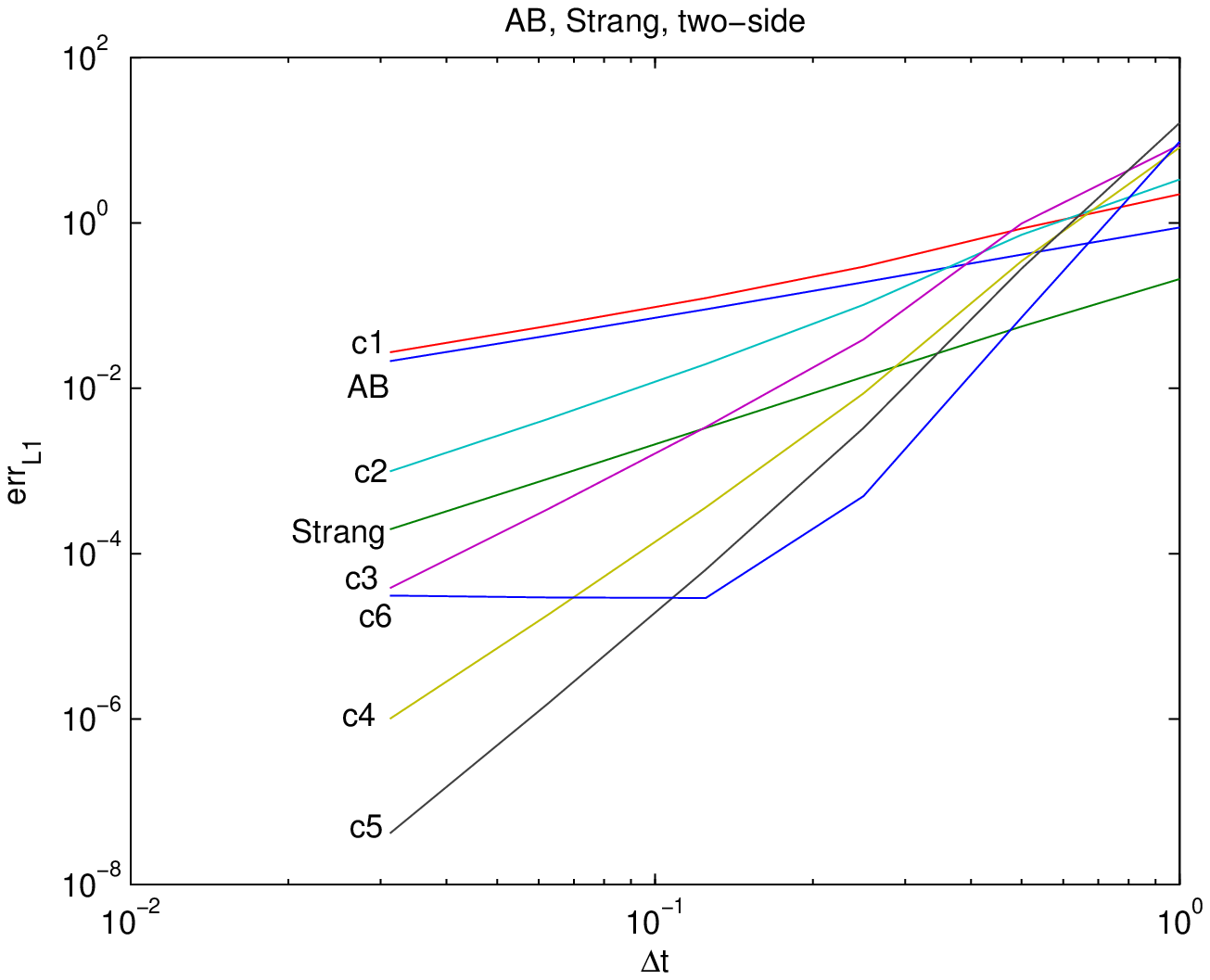} 
\end{center}
\caption{\label{two_phase} Numerical errors of the one-side Splitting scheme with $A$ (upper figure),  the one-side Splitting scheme with $B$ (middle figure) and the
iterative schemes with $1, \ldots, 6$ iterative steps (lower figure).}
\end{figure}

\begin{remark}
For all iterative schemes, we can reach faster results as for the
The iterative schemes with fast computations of the exponential matrices
standard schemes.
With $4-5$ iterative steps we obtain more accurate results as we did
for the expensive standard schemes.
With one-side iterative schemes we reach the best convergence results.
\end{remark}

\section{Conclusions and Discussions }
\label{concl}

We present the coupled model for a transport model
for deposition species in a plasma environment.
We assume the flow field is computed by the plasma 
model and the transport of the deposition species 
with a transport-reaction model.

Such a first model can help to understand the 
important modeling of the plasma environment in 
a CVD reactor.

\bibliographystyle{plain}

\end{document}